\def\update{08/04/2025} 

\documentclass[11pt]{epiarticle}
\usepackage{epimath}
 
\setpapertype{A4}

\usepackage[english]{babel}

\usepackage[dvips]{graphicx} 

\newtheorem{conjecture}[theorem]{Conjecture} 
\title{Variations on Schanuel's Conjecture \\ for elliptic and quasi--elliptic functions\\ I: the split case}
\titlemark{Variations on Schanuel's Conjecture \hfill \update}
\author{Cristiana Bertolin \&\ Michel Waldschmidt}
\authormark{Cristiana Bertolin \&\ Michel Waldschmidt}
\date{\today} % Empty date or tweak it according to your needs
\journal{Hardy-Ramanujan Journal -- (yyyy), ---} % Epijournal name
\acceptation{submitted dd/mm/yyyy, accepted dd/mm/yyyy, revised \update}
\newcommand{\details}[1]{}

\newcommand\audessus[2]{\genfrac{}{}{0pt}{}{#1}{#2}}

\newcommand{\Z}{\mathbb{Z}}
\newcommand{\Q}{\mathbb{Q}} 
\newcommand{\C}{\mathbb{C}}
\newcommand{\G}{\mathbb{G}}
\renewcommand{\AA}{\mathbb{A}}
\newcommand{\PP}{\mathbb{P}} 

\newcommand{\rme}{\mathrm {e}}
\newcommand{\rmi}{\mathrm {i}}

\newcommand{\End}{\mathrm{End}}

\newcommand{\Lie}{\mathrm{Lie}\,}

\newcommand{\Qbar}{\overline{\Q}}

\newcommand{\cE}{\mathcal{E}}

\begin{document}
 
\maketitle

\thanks{ 
The first author is supported by the project funded by the European Union – Next Generation EU under the National Recovery and Resilience Plan (NRRP), Mission 4 Component 2 Investment 1.1 - Call PRIN 2022 No. 104 of February 2, 2022 of Italian Ministry of University and Research; Project 20222B24AY (subject area: PE - Physical Sciences and Engineering) ``The arithmetic of motives and L-functions".
}
 
\begin{prelims}

\def\abstractname{Abstract}
\abstract{
 It is expected that Schanuel's Conjecture contains all ``reasonable" statements that can be made on the values of {\em the exponential function}. In particular it implies the Lin\-de\-mann-Weierstrass Theorem and the Conjecture on algebraic independence of logarithms of algebraic numbers.

Our goal is to state conjectures {\em \`a la Schanuel}, which imply conjectures {\em \`a la Lindemann-Weierstrass}, for the exponential map of an extension $G$ of an elliptic curve $\cE$ by the multiplicative group $\G_m$. In the present paper we assume that the extension is split, that is $G=\G_m\times \cE$. In a second paper in preparation we will deal with the non--split case, namely when the extension is not a product. Here we propose the {\em split semi--elliptic Conjecture}, which involves the exponential function and the Weierstrass $\wp$ and $\zeta$ functions, related with integrals of the first and second kind. In the second paper, our {\em non-split semi--elliptic Conjecture} will also involve Serre's functions, related with integrals of the third kind.
 
 We expect that our conjectures contain all ``reasonable" statements that can be made on the values of these functions.

In the present paper we highlight the geometric origin of the split semi--elliptic Conjecture: it is {\em equivalent to} the Grothendieck-Andr\'{e} generalized period Conjecture applied to the 1-motive 
\(
M=[u:\Z \rightarrow \G_m^s \times \cE^n ], 
\) 
which is the Elliptico--Toric Conjecture of the first author. 

We show that our split semi--elliptic Conjecture implies three theorems of Schneider on elliptic analogs of the Hermite--Lindemann 
and Gel'fond--Schneider's theorems, as well as a conjecture on the Weierstrass zeta function. 
}



% Add table of contents (optional)
\tableofcontents

\end{prelims}

%%%%%%%%%%%%%%%%%%%%%
% Content begins here
%%%%%%%%%%%%%%%%%%%%%

%-----------------------------------------
\section{Introduction}%\label{S:Introduction}

Let $\Qbar$ be the algebraic closure in $\C$ of the field $\Q$ of rational numbers. Since the transcendence degree of $\Qbar$ over $\Q$ is $0$, complex numbers are algebraically independent over $\Q$ if and only if they are algebraically independent over $\Qbar$ -- in this case we simply say that they are algebraically independent. When we speak of the transcendence degree (denoted $\mathrm{tran.deg}$) without specifying over which field, it means that it is over $\Q$. In the same vein, we say that a number is algebraic (resp. transcendental) if it is algebraic (resp. transcendental) over $\Q$, which means that it belongs (or not) to $\Qbar$. 

Hermite proved the transcendence of the number $\rme$ in 1873, 
Lindemann the transcendence of $\pi$ in 1882. These two results are special cases of the following statement, known as the {\em Hermite--Lindemann Theorem}: 

\begin{theorem}
[HL] % \label{Theorem:HL}
If $\alpha$ is a non--zero complex number, then one at least of the two numbers $\alpha$, $\rme^\alpha$ is transcendental. 
\end{theorem}

The following more general statement is known as the {\em Lindemann--Weierstrass Theorem} (1885): 

\begin{theorem}
[LW]
 \label{Theorem:LW} 
If $\alpha_1,\dots,\alpha_s$ are $\Q$--linearly independent algebraic numbers, then the numbers $\rme^{\alpha_1},\dots,\rme^{\alpha_s}$ are algebraically independent.
\end{theorem}

\par\noindent This theorem is a consequence of the following conjecture proposed by Schanuel \cite{LangITN}, which is supposed to contain all ``reasonable" statements that can be made on the values of the exponential function:

\begin{conjecture} 
[Schanuel's Conjecture] \label{SchanuelConjecture}
 If $t_1,\dots,t_s$ are $\Q$--linearly independent complex numbers, then at least $s$ of the $2s$ numbers $t_1,\dots,t_s, \rme^{t_1},\dots, \rme^{t_s}$ are algebraically independent.
\end{conjecture}

\medskip

 In other terms the transcendence degree over $\Q$ of the field $\Q(t_1,\dots,t_s, \rme^{t_1},\dots, \rme^{t_s})$ is at least $s$. 
This lower bound $s$ is clearly the best possible: for instance if $t_1,\dots,t_s$ are algebraic, then the transcendence degree is $s$ (and this is LW Theorem \ref{Theorem:LW}). If $ \rme^{t_1},\dots, \rme^{t_s}$ are algebraic numbers, say $\alpha_1,\dots,\alpha_n$, writing $t_\ell=\log\alpha_\ell$ ($\ell=1,\dots,s$), Schanuel's Conjecture \ref{SchanuelConjecture} reduces to the conjecture on the algebraic independence of $\Q$--linearly independent logarithms of algebraic numbers. 

 The previous statements deal with {\em the exponential function} of the torus $\G_{m}$
 \[
\begin{aligned} 
\exp: \Lie \G_{m \, \C} & \longrightarrow \G_m(\C) = \C^\times \\
\nonumber t & \longmapsto \rme^t. 
\end{aligned}
\] 

In \cite[\S 3]{B02} the first author proves that Schanuel's Conjecture \ref{SchanuelConjecture} is {\em equivalent to} the Grothendieck-Andr\'{e} generalized period Conjecture applied to the 1-motive 
\[
M=[u:\Z \rightarrow \G_m^s ], \; u(1)=( \rme^{t_1}, \dots, \rme^{t_s} ) \in \G_m^s (\C)
\] 
(see 
\cite[Letter of Y. Andr\'{e}]{B20} for a nice explanation of the Grothendieck-Andr\'{e} generalized period Conjecture).

 \medskip
 
In the present paper we state a conjecture {\em \`a la Schanuel}, which we call the {\em Split Semi--Elliptic Conjecture}, for a split extension of an elliptic curve $\cE$ by the multiplicative group: $\G_m\times \cE$. Based on \cite{B02}, we relate our conjecture with the motivic conjecture of Grothendieck--Andr\'e. In a forthcoming paper \cite{BWFutur}, we will deal with the case of non-split extensions, and relate it to the Grothendieck--Andr\'e conjecture thanks to \cite{B02,BFutur}, in connection with \cite{W_Schanuel}. 

Section \ref{S:The Split Semi--Elliptic Conjecture}~is devoted to the statement of the main Conjecture \ref{Conjecture:SplitSemiElliptic}. Section \ref{S:EllipticAndQuasiElliptic}~contains the necessary background on elliptic and quasi--elliptic functions, including addition, multiplication and division formulae.

The geometric origin of our Split Semi--Elliptic Conjecture is explained in section \ref{S:Elliptico--toric}, where we prove its equivalence with the Elliptico--Toric Conjecture \ref{Conjecture:elliptico--toric} of \cite[\S 1]{B02}. Proofs of some consequences of Conjecture \ref{Conjecture:SplitSemiElliptic} are given in section \ref{S:ConsequencesSplitElliptic}

\section{The Split Semi--Elliptic Conjecture}\label{S:The Split Semi--Elliptic Conjecture}

Let $\cE$ be an elliptic curve over $\C$. 
We use a Weierstrass parametrization of the exponential of $\cE$ with a Weierstrass $\wp$ function; we denote by $\Omega$ the lattice of periods of $\wp$, by $\zeta$ the Weierstrass zeta function associated to $\Omega$ and by $g_2$ and $g_3$ the invariants of $\wp$. Let $(\omega_1,\omega_2)$ be a pair of fundamental periods of $\Omega$:
\[
\Omega=\Z\omega_1+\Z\omega_2
\]
and let $k$ be the field of endomorphisms of $\cE$, namely $k:= \End (\cE) \otimes_\Z \Q$:
\[
k=\begin{cases}
\Q & \text{ in the non--CM case,}
\\
\Q(\tau) & \text{ in the CM case}
\end{cases}
\]
where $\tau:=\omega_2/\omega_1$. 

The exponential map of the product $ \G_m\times \cE$ (composed with a projective embedding) is 
\[
\begin{aligned}
 \exp_{\G_m\times \cE }: \Lie ( \G_m\times \cE )_\C =\C^2 & \longrightarrow ( \G_m\times \cE )_\C(\C) \subset (\AA\times \PP^{2} )(\C)\\
 \nonumber (t,z) & \longmapsto \big(\rme^t, \sigma(z)^3 [\wp(z) : \wp'(z) : 1] \big)
\end{aligned}
\] 
where $\AA$ is the affine line and $\PP^2$ the projective space of dimension $2$.

In terms of meromorphic functions, the four functions $z$, $\rme^z$, $\wp(z)$, $\zeta(z)$ are related with the parametrization of the exponential map of the algebraic group $\widetilde{G}$ of dimension $4$, which is the product of the additive group $\G_a$, the multiplicative group $\G_m$ and the non--trivial extension of $\cE$ by $\G_a$ (see \cite[Exercise 20.102]{M}): 
\[
\begin{matrix}
 \exp_{\widetilde{G}}:&\Lie( \widetilde{G}) _{\C}= \C^4 &\longrightarrow& \widetilde{G} (\C) \subset (\AA^2 \times \PP^{4})(\C)
 \hfill
\\
 &(z_0,z_1,w,z) & \longmapsto &\big((z_0,\rme^{z_1}),\hfill
 \\
 && &
\sigma(z)^3 [\wp(z) : \wp'(z) :1:
 w+\zeta(z):(w+\zeta(z))\wp'(z)+2\wp(z)^2] \big),
 \end{matrix}
\] 
where $\AA^2$ is the affine space of dimension $2$ and $\PP^4$ the projective space of dimension $4$.

The neutral element has coordinates $(0,1)[0:1:0:0:0]$. Under $\exp_{\widetilde{G}}$, the subgroup $\G_a^2$ of $ \widetilde{G}$ is the image of the subspace $\C\times\{0\}\times\C\times\{0\}$, 
 the subgroup $\G_m$ of $ \widetilde{G}$ is the image of the subspace $\{0\}\times\C\times\{0\}\times\{0\}$, 
 the kernel of $\exp_{\widetilde{G}}$ is the image of 
\[
\bigl\{
(0,2\pi\rmi a, -\eta(\omega),\omega)\; \mid\; a\in\Z,\; \omega\in\Omega
\bigr\}
\]
which is a discrete subgroup of rank $3$ of $\C^4$, and the torsion subgroup of $\widetilde{G}$
 is the image of the subgroup 
 \[
\bigl\{
(0,2\pi\rmi a/m, -\eta(\omega)/m,\omega/m)\; \mid\; a\in\Z,\; \omega\in\Omega, \; m\geqslant 1
\bigr\}\subset\C^4.
\]  
 
For the algebraic groups $\G_m\times\cE$ and $\widetilde{G}$, we suggest the following conjecture {\em \`a la Schanuel}.

 \begin{conjecture} 
[Split Semi--Elliptic Conjecture] \label{Conjecture:SplitSemiElliptic}
Let $\Omega$ be a lattice in $\C$. Let $t_1, \dots,t_s$ be $\Q$--linearly independent complex numbers and $p_1,\dots,p_n$ be $k$--linearly independent complex numbers in $\C \smallsetminus \Omega$.
Then the transcendence degree of the field 
\begin{equation}\label{Equation:K}
K:=\Q\bigl(t_1, \dots,t_s, e^{t_1},\dots ,e^{t_s}, g_2,g_3, p_1, \dots,p_n, \wp(p_1), \dots,\wp(p_n), \zeta(p_1), \dots,\zeta (p_n)\bigr)
\end{equation} 
is at least $s+2n$, unless $2\pi\rmi\Q\subset \Q t_1+\cdots+\Q t_s$ and $\Omega\subset k p_1+\cdots+kp_n$, in which case it is at least $s+2n-1$.
\end{conjecture}

We will call {\em exceptional} the case where $2\pi\rmi\Q\subset \Q t_1+\cdots+\Q t_s$ and $\Omega\subset k p_1+\cdots+kp_n$. As a matter of fact, we will see in section \ref{S:Elliptico--toric} that this exceptional case implies the non--exceptional cases. 

The lower bound for the transcendence degree cannot be improved:
\\
$\bullet$ 
In the exceptional case, the transcendence degree of $K$ is at most $s+2n-1$ for instance when the $s+n+2$ numbers $\rme^{t_1}, \dots,\rme^{t_s}, g_2,g_3, \wp(p_1),\dots,\wp(p_n)$ are all algebraic with $t_1=2\pi\rmi$, and further $p_1=\omega_1/2$, $p_2=\omega_2/2$ in the non--CM case, $p_1=\omega_1/2$ in the CM case (this follows from Legendre relation \eqref{Equation:Legendre} and relation \eqref{Equation:kappa} -- see Remark \ref{Remarque:GGPCEllipticCase}).
\\
$\bullet$ 
In the non--exceptional case, the transcendence degree of $K$ is at most $s+2n$ for instance when the $s+n+2$ numbers $t_1, \dots,t_s,g_2,g_3, p_1,\dots, p_n$ are all algebraic. In this case we have $2\pi\rmi\not\in \Q t_1+\cdots+\Q t_s$ and $\Omega\cap(k p_1+\cdots+kp_n)=\{0\}$ and so this case is not exceptional.
 
 \medskip
There are cases where the upper bound for the transcendence degree is not optimal: a consequence of \cite[Corollary 2.3]{W_Schanuel} is that given arbitrary $g_2,g_3$ with $g_2^3\not=27g_3^2$, for almost all $s+n$ tuples $( t_1, \dots,t_s,p_1,\dots,p_n)$, the $2s+3n$ numbers 
\[
t_1, \dots,t_s, e^{t_1},\dots ,e^{t_s}, p_1, \dots,p_n, \wp(p_1), \dots,\wp(p_n), \zeta(p_1), \dots,\zeta (p_n)
\] 
are algebraically independent over $\Q(g_2,g_3, \omega_1,\omega_2,\eta_1,\eta_2)$. 

In Theorem \ref{G<=>SplitSchanuel}, 
we prove that Conjecture \ref{Conjecture:SplitSemiElliptic} is equivalent to 
the Grothendieck-Andr\'{e} generalized period Conjecture applied to the 1-motive 
\[
 M=[u:\Z \rightarrow \G_m^s \times \cE^n ], \; u(1)=(\rme^{t_1}, \dots, \rme^{t_s} , P_1, \dots, P_n) \in (\G_m^s \times \cE^n )(\C),
\] 
 where $ P_i=[\wp(p_i): \wp'(p_i):1] $ for $i=1, \dots,n$, which is the Elliptico--Toric Conjecture (Conjecture \ref{Conjecture:elliptico--toric} below; see \cite[\S 1]{B02} and \cite[Letter of Y. Andr\'{e}]{B20}).

Schanuel's Conjecture \ref{SchanuelConjecture} is the case $n=0$ of Conjecture \ref{Conjecture:SplitSemiElliptic} when there is no elliptic curve.
In this case $\Omega\cap(k p_1+\cdots+kp_n)=\{0\}$.

Here is the case $s=0$ of Conjecture \ref{Conjecture:SplitSemiElliptic} when there is no $G_m$ factor (see also \cite[Conjecture 4.2]{BPSS}). 
 In this case $2\pi\rmi\not\in\Q t_1+\cdots+\Q t_s$. 
 
 \begin{conjecture}
[Elliptic Schanuel Conjecture] 
\label{Conjecture:EllipticSchanuel}
 Let $\Omega$ be a lattice in $\C$. Let $p_1,\dots,p_n$ be $k$--linearly independent elements in $\C \smallsetminus \Omega$. Then at least $2n$ of the $2+3n$ numbers 
\[
g_2,g_3,p_1, \dots,p_n, \wp(p_1), \dots,\wp(p_n), \zeta(p_1), \dots,\zeta (p_n) 
\]
are algebraically independent.
\end{conjecture}

\medskip
 If we assume the lattice $\Omega$ to have algebraic invariants $g_2$ and $g_3$ and $t_1, \dots,t_s$, $p_1,\dots,p_n$ to be algebraic, from Conjecture \ref{Conjecture:SplitSemiElliptic}, we get a conjecture \`{a} la Lindemann--Weierstrass.
 
\begin{conjecture}[Split Semi--Elliptic LW Conjecture] 
\label{Conjecture:SplitSemiEllipticLW}
 Let $\Omega$ be a lattice in $\C$ with algebraic invariants $g_2,g_3$. If 
 \begin{itemize}
\item $t_1, \dots,t_s$ are $\Q$--linearly independent algebraic numbers,
\item $p_1,\dots,p_n$ are $k$--linearly independent algebraic numbers, 
 \end{itemize} 
 then the $s+2n$ numbers 
 $e^{t_1},\dots ,e^{t_s}, \wp(p_1), \dots,\wp(p_n), \zeta(p_1), \dots,\zeta (p_n) $
 are algebraically independent.
\end{conjecture}

Lindemann's Theorem on the transcendence of $\pi$ implies that when $t_1, \dots,t_s$ are algebraic numbers, then $2\pi\rmi\not\in\Q t_1+\cdots+\Q t_s$.
Also, Corollary \ref{corollary:Schneider} of Schneider's Theorem \ref{Th:Schneider}.1 implies that the poles $\not=0$ of a Weierstrass $\wp$ function with algebraic invariants are transcendental; consequently, when $p_1,\dots,p_n$ are algebraic numbers, then $\Omega\cap(k p_1+\cdots+kp_n)=\{0\}$.

 Two cases of Conjecture \ref{Conjecture:SplitSemiEllipticLW} are known: 
 the LW Theorem \ref{Theorem:LW} and the Phi\-lip\-pon--W\"ustholz Theorem, which states that {\em the values of a Weierstrass $\wp$ function, with algebraic invariants and with complex multiplication, at $k$--linearly independent algebraic numbers, are algebraically independent } (see \cite[Corollaire 0.3]{P83} and \cite[Korollar 2]{W83}).

Another consequence of Conjecture \ref{Conjecture:SplitSemiElliptic} is the following {\em Conjecture on algebraic independence of semi--elliptic logarithms of algebraic points in the split case}, which contains the conjectures on algebraic independence of $\Q$--linearly independent (ordinary) logarithms of algebraic numbers (a special case of Schanuel's Conjecture), and on algebraic independence of $k$--linearly independent elliptic logarithms of algebraic points (a special case of Conjecture \ref{Conjecture:EllipticSchanuel}).
Very few unconditional results are known, including the six exponentials Theorem and some elliptic analogs (see for instance \cite[Chap. II]{LangITN} and \cite{Ramachandra}).

\begin{conjecture}
[Split Semi--Elliptic Logarithms Conjecture] 
\label{Conjecture:SplitSemiEllipticLogarithms}
Let $\Omega$ be a lattice in $\C$ with algebraic invariants $g_2,g_3$.
Let $s$, $m$, $n$ be non negative integers with $0\leqslant m\leqslant n$. Let
 \begin{itemize}
\item $t_1, \dots,t_s$ be $\Q$--linearly independent complex numbers such that the numbers $\alpha_\ell:=\rme^{t_\ell}$ ($\ell=1,\dots,s$) are algebraic; write $t_\ell=\log\alpha_\ell$;
\item $p_1,\dots,p_n$ be $k$--linearly independent elements of $\C\smallsetminus\Omega$ such that the numbers $\wp(p_i)$ $(1\leqslant i\leqslant m$) and $\zeta(p_j)$ $(m+1\leqslant j\leqslant n$) are algebraic. 
 \end{itemize} 
 Then the $s+n$ numbers 
 \[
 \log\alpha_1,\dots,\log\alpha_s,\; p_1,\dots,p_n
 \]
 are algebraically independent.
 %\end{quote}
\end{conjecture}

We deduce Conjecture \ref{Conjecture:SplitSemiEllipticLogarithms} from Conjecture \ref{Conjecture:SplitSemiElliptic} in section \ref{S:ConsequencesSplitElliptic} We complete this section with some more consequences of Conjecture \ref{Conjecture:SplitSemiElliptic}, proofs of which will also be given in section \ref{S:ConsequencesSplitElliptic}
 
In \cite[Satz II]{Schneider} and \cite[Zweites Kapitel, \S 4,Satz 15]{SchneiderLivre} (see also \cite[Chap.~20]{M}), assuming that the invariants $g_2$ and $g_3$ are algebraic, Schneider proves three theorems on the common algebraic values of two algebraically independent functions, namely ($\alpha z+\beta \zeta(z), \wp(z)$), $(\rme^{\alpha z},\wp(z)$) and ($\wp(z), \wp^\star(z)$), when $\wp$ and $\wp^\star$ are two algebraically independent Weierstrass $\wp$ functions with algebraic invariants. In the next statement for the third case we restrict ourselves to $\wp^\star(z)=c^2\wp(cz)$ with $c\in\C^\times$. 

\begin{theorem} [Schneider] \label{Th:Schneider} 
\hfill 
\begin{enumerate}
\item
Let $\alpha$, $\beta$ be two complex numbers with $(\alpha,\beta)\not=(0,0)$ and let $p\in\C \smallsetminus\Omega$. Then one at least of the six numbers 
$\alpha$, $\beta$, $g_2$, $g_3$, $\wp(p)$, $\alpha p+\beta\zeta(p)$ is transcendental. 
\item
Let $p\in\C \smallsetminus\Omega$ and let $\alpha$ be a nonzero complex number. Then one at least of the five numbers $\alpha$, $g_2$, $g_3$, $\rme^{\alpha p}$, $\wp(p)$ is transcendental. 
\item
Let $\alpha$ and $p $ be two complex number such that $\alpha\not\in k$, $p\not\in \Omega$ and $\alpha p\not\in \Omega$. Then 
one at least of the five numbers $\alpha$, $g_2$, $g_3$, $\wp(p)$, $\wp(\alpha p)$ is transcendental. 
\end{enumerate}
\end{theorem}
 
An equivalent formulation to Theorem \ref{Th:Schneider}.1 is the following statement: {\em if $g_2$, $g_3$ and $\wp(p)$ are algebraic, then the three numbers $1,p,\zeta(p)$ are linearly independent over the field of algebraic numbers. } In particular for $p \in \C \smallsetminus \Omega$, one at least of the four numbers $g_2$, $g_3$, $p$, $\wp(p)$ is transcendental and one at least of the four numbers $g_2$, $g_3$, $\wp(p)$, $\zeta(p)$ is transcendental. This last statement implies the transcendence of one at least of the three numbers $g_2,g_3,\eta(\omega)$ when $\omega$ is a nonzero period (recall $\wp'(\omega/2)=0$ and $\zeta(\omega/2)=\eta(\omega)/2$ when $\omega\in\Omega$ and $\omega/2\not\in\Omega$). 
\medskip

Here is a corollary to Theorem \ref{Th:Schneider}.

\begin{corollary} \label{corollary:Schneider}
 Let $p \in \C \smallsetminus \Omega$. Assume $g_2$, $g_3$ and $\wp(p)$ are algebraic. Then $p$ is transcendental. 
 \\
 Further, let $\alpha\in\Qbar\smallsetminus k$. Then $\alpha p\not\in\Omega$ and $\wp(\alpha p)$ is transcendental.
 \end{corollary}
 
 \begin{proof}[Proof of Corollary \ref{corollary:Schneider}]
 The fact that $p$ is transcendental follows from Theorem \ref{Th:Schneider}.1 with $\alpha=1$ and $\beta=0$. 
 
 Assume $\omega:=\alpha p\in\Omega\smallsetminus\{0\}$. Let $n\geqslant 2$ be such that $\omega/n\not\in\Omega$. Set $\alpha'=\alpha/n$. Then $\alpha'p=\omega/n\not\in\Omega$ and since $\wp(\alpha' p)=\wp(\omega/n)$ is algebraic (this is a torsion point --- see Lemma \ref{Lemma:division}), the five numbers $g_2$, $g_3$, $\alpha'$, $\wp(p)$, $\wp(\alpha' p)$ are algebraic, which contradicts Theorem \ref{Th:Schneider}.3. 
 
 Hence $\alpha p\not\in\Omega$. From Theorem \ref{Th:Schneider}.3 we deduce that $\wp(\alpha p)$ is transcendental. 
\null\hfill$\square$
\end{proof}
 
 \medskip
 
\par\noindent 
The assumption that $\alpha$ does not belong to $k$ is necessary: if $\alpha$ belongs to $k$ and if $\wp(p)$ is algebraic over $\Q(g_2,g_3)$, then also $\wp(\alpha p)$ is algebraic over $\Q(g_2,g_3)$ (see section \ref{S:EllipticAndQuasiElliptic}).
 
In \cite[\S 5]{K} the second author proposes an analog to Schneider's Corollary \ref{corollary:Schneider} for the Weierstrass zeta function. The next remark shows that this conjecture needs to be modified, in order to eliminate two special cases, associated with elliptic curves having nontrivial automorphisms. 

\begin{remark}\label{Remark:specialcases} \rm
\hfill\break
1. Assume $g_3=0$, which means that the elliptic curve $\cE$ is a CM curve with field of endomorphisms $\Q(\rmi)$. When $\alpha$ satisfies $\alpha^4=1$, then 
\[
\wp(\alpha z)=\alpha^2\wp(z),\quad \zeta(\alpha z)=\alpha^3\zeta(z).
\] 
2. 
 Assume $g_2=0$, which means that the elliptic curve $\cE$ is a CM curve with field of endomorphisms $\Q(\zeta_6)$ where $\zeta_6$ is a primitive root of unity of order $6$. When $\alpha$ satisfies $\alpha^6=1$, then 
\[
\wp(\alpha z)=\alpha^4\wp(z),\quad \zeta(\alpha z)=\alpha^5\zeta(z).
\] 
\end{remark}
In each of these two cases, when $\zeta(p)$ is algebraic, then $\zeta(\alpha p)$ is also algebraic. 

\begin{conjecture} 
[$\zeta$--Conjecture]
 \label{Conjecture:zeta}
 Assume that $g_2$ and $g_3$ are algebraic. 
Let $p \in \C \smallsetminus \Omega$. Assume that $\zeta(p)$ is algebraic. Let $\alpha$ be an algebraic number different from $0,1,-1$. Assume also that $\alpha^4\not=1$ if $g_3=0$ and $\alpha^6\not=1$ if $g_2=0$. Then $\alpha p\not\in\Omega$ and $\zeta( \alpha p)$ is transcendental.
\end{conjecture}

The only special case of Conjecture \ref{Conjecture:zeta} which is known so far is \cite[Corollary 2.1]{K}: {\em assume $\wp(p)$ is transcendental. Then, with at most one exception, for $r$ a positive rational number, the number $\zeta(rp)$ is defined and is transcendental over the field $\Q(g_2,g_3)$.}
 
In Propositions \ref{G=>S} and \ref{G=>W}, we show that Conjecture \ref{Conjecture:SplitSemiElliptic} implies Schneider's Theorem \ref{Th:Schneider} and the $\zeta$--Conjecture \ref{Conjecture:zeta}.

%-----------------------------------------
\section{Elliptic and quasi--elliptic functions}\label{S:EllipticAndQuasiElliptic}
 
Consider the following Weierstrass functions attached to our lattice $\Omega = \Z \omega_1 + \Z \omega_2$ in $\C$ having elliptic invariants $g_2,g_3$. 

\begin{itemize}
 
 \item The canonical product of Weierstrass associated with $\Omega$ is the sigma function
 \[
 \sigma(z)=z\prod_{\omega\in\Omega \smallsetminus\{0\}}\left(1-\frac z \omega\right) \rme^{ \frac z \omega + \frac {z^2}{2\omega^2}}.
 \]
 \item 
 The Weierstrass zeta function is the logarithmic derivative of the sigma function: $\zeta=\sigma'/\sigma$.
 \item 
 The Weierstrass elliptic function is $\wp=-\zeta' $.
\end{itemize}

\subsection{Periodicity}
The periods of the Weierstrass elliptic function $\wp$ are elliptic integrals of the first kind. 
The Weierstrass zeta function $\zeta$ has quasi--periods $\eta(\omega)=\zeta(z+\omega)-\zeta(z)$ which are given by elliptic integrals of the second kind. 
 
The periodicity relations satisfied by the Weierstrass functions are, for $\omega\in\Omega$, 
\\
$\bullet$ \quad
$\wp(z+\omega)=\wp(z)$,
\\
$\bullet$ \quad 
$\zeta(z+\omega)=\zeta(z)+\eta(\omega) $,
\\
$\bullet$ \quad $
\sigma(z+\omega)=\epsilon(\omega)\sigma(z) \rme^{\eta(\omega)(z+\frac \omega 2)}
$ with $\epsilon(\omega)=1$ if $\omega/2\in\Omega$, $\epsilon(\omega)=-1$ if $\omega/2\not\in\Omega$.
 
 The quasi--periodicity of the Weierstrass zeta function defines a $\Z$--linear map 
\[
\begin{aligned}
\eta: \; & \Omega \to \C, 
\\
&\omega \mapsto \eta(\omega). 
\end{aligned}
\] 
We set $\eta_i= \eta(\omega_i)$ for $i=1,2$. 

 {\em Legendre relation} plays an important role: 
\begin{equation}\label{Equation:Legendre}
 \omega_2 \eta_1 - \omega_1 \eta_2 = 2\pi\rmi,
\end{equation}
 with the $+$ sign when the imaginary part of $\omega_2/\omega_1$ is positive --- cf. \cite[(20.411)]{WW}, \cite[Chap.~IV \S 2 Th.2]{Cha}, \cite[Exercise 20.33]{M}.

Let
$\cE$ be the elliptic curve over $\C$ defined by the lattice $\Omega$. 
Recall that if $\cE$ has no complex multiplication, the field of endomorphisms $k$ of $\cE$ is the field $\Q$ of rational numbers, while if $\cE$ has complex multiplication, then $k$ is an imaginary quadratic extension of $\Q$, more precisely $ k=\Q(\tau)$ with $\tau :=\frac{\omega_2}{\omega_1}$. In both cases, $\Q \subseteq k \subseteq \Qbar$. 

For convenience, we use small letters for elliptic logarithms of points in $\cE(\C)$ which are written with capital letters: $\exp_{\cE}(p)=P \in \cE (\C)$ for $p \in \Lie \cE_\C$.

 %-----------------------------------------
 \subsection{Addition and multiplication formulae for Weierstrass functions} %\label{SS:Formulae}
 
 In this section we sum up addition and multiplication formulae for Weierstrass functions.
 
The following addition formulae are well known. Among many references are
\cite[I.1 formulae (8), (9), (10) p.~159-160]{F}, 
\cite[Chap.~XX, 20.31 page 441 and 
20.53 page 451]{WW}\footnote{Beware that $(2\omega_1,2\omega_2)$ denotes a pair of fundamental periods of $\cE$ in \cite{WW}.}, 
 \cite[Chap.~1]{LangECDA},
 \cite[Chap. 3]{Silverman},
\cite[Chap.~III \S 4]{Cha},
 \cite[Chap.~1 \S 3]{LangEC},
 \cite[Chap. 3 \S 10]{Cox},
 \cite[Chap. 9]{Washington}, 
 \cite[Chap.~20]{M}.
 
 \bigskip 
 
 \begin{tabular}{ |p{15cm}| } 
\hline
\multicolumn{1}{|c|}{} \\
\multicolumn{1}{|c|}{\textbf{Addition Formulae}} \\
\multicolumn{1}{|c|}{} \\
\hline \rule[-4mm]{0cm}{1cm}
$\wp(z_1+z_2) = \frac{1}{4} \Big( \frac{\wp'(z_1)-\wp'(z_2)}{\wp(z_1)-\wp(z_2)}\Big)^2 - \wp(z_1)-\wp(z_2) \in \Q(g_2,g_3,\wp(z_1),\wp(z_2),\wp'(z_1),\wp'(z_2))$\\
\hline \rule[-4mm]{0cm}{1cm}
$\zeta(z_1+z_2) - \zeta(z_1)- \zeta(z_2) = \frac{1}{2} \frac{\wp'(z_1)-\wp'(z_2)}{\wp(z_1)-\wp(z_2)}\in \Q(g_2,g_3,\wp(z_1),\wp(z_2),\wp'(z_1),\wp'(z_2))$ \\ 
\hline \rule[-4mm]{0cm}{1cm}
$ \frac{\sigma (z_1+z_2) \sigma (z_1-z_2)}{\sigma (z_1)^2\sigma (z_2)^2}= \wp(z_2)-\wp(z_1) \in \Q(g_2,g_3,\wp(z_1),\wp(z_2))$ \\ 
\hline
 \end{tabular}
 \begin{center}
{\bf Table 1}
 \end{center}
 
 \medskip
 \medskip
 
 For the multiplication by a rational number, in this section we prove the following formulae, where, for any integers $n,m$ with $m\not=0$, $f_{n,m}(z)$ is a function in $\Q(g_2,g_3,\wp(\frac{z}{m}), \wp'(\frac{z}{m}))$ which is constant if and only if $n = \pm m$ or $n=0$. 
 \medskip
 \medskip
 
 \begin{tabular}{ |p{15cm}| }
\hline
\multicolumn{1}{|c|}{} \\
\multicolumn{1}{|c|}{\textbf{ Multiplication by $\frac{n}{m}$ formulae, with $n,m$ integers, $m\not=0$}} \\
\multicolumn{1}{|c|}{} \\
\hline \rule[-4mm]{0cm}{1cm}
$\wp \big(\frac{nz}{m}\big) = \wp(z) - \frac{f''_{n,m}(z)f_{n,m}(z)-f'_{n,m}(z)^2}{n^2 f_{n,m}(z)^2} $\\
\hline \rule[-4mm]{0cm}{1cm}
$\zeta\big(\frac{nz}{m}\big) = \frac{n}{m} \zeta(z) + \frac{f'_{n,m}(z)}{mn f_{n,m}(z)}$ \\ 
\hline \rule[-4mm]{0cm}{1cm}
$ \sigma\big(\frac{nz}{m}\big)^{m^2} = \sigma (z)^{n^2} f_{n,m}(z)$ \\ 
\hline
 \end{tabular}
 \begin{center}
{\bf Table 2}
 \end{center}
 
 \medskip
 \medskip

\begin{proof}[Proof of the formulae in Table 2]
These formulae are special cases of \cite[Proposition 5]{Reyssat82} (his number $\tau$ is a quotient of two nonzero periods, the case of a rational number is included). See also \cite[I.2 p.184 and I.3 p.210]{F}. For the convenience of the reader we include a sketch of proofs. 

For $n\in\Z$, the meromorphic function 
\[
f_{n,1}(z):=\frac {\sigma(nz)}{\sigma(z)^{n^2}}
\]
belongs to $\Q(g_2,g_3,\wp(z),\wp'(z))$ (see for instance \cite[\S 1]{Wald79}). 

For $n$ and $m$ in $\Z$ with $m>0$ define
\begin{equation}\label{def:f_{n,m}}
f_{n,m}(z):=\frac {\sigma(nz/m)^{m^2}}{\sigma(z)^{n^2}}
\cdotp
 \end{equation}
 We have  
 \[
 f_{-n,m}(z)=(-1)^m f_{n,m}(z).
 \]
Since 
 \[
 f_{n,m}(mz) = \frac{ \sigma(nz)^{m^2}}{\sigma(mz)^{n^2}} = \frac{ \sigma(nz)^{m^2}}{\sigma(z)^{n^2m^2}} \cdot \frac{ \sigma(z)^{n^2m^2}}{\sigma(mz)^{n^2}} = \frac{f_{n,1}(z)^{m^2}}{f_{m,1}(z)^{n^2}},
 \]
from the periodicity with respect to $\Omega$ of the function $f_{n,1}(z)$ we deduce, for any period $\omega$, 
 \[ 
 f_{n,m}(z+m \omega) = f_{n,m}\left(m\left(\frac{z}{m}+ \omega\right) \right)= \frac{f_{n,1}(\frac{z}{m}+ \omega)^{m^2}}{f_{m,1}(\frac{z}{m}+ \omega)^{n^2}} = \frac{f_{n,1}(\frac{z}{m})^{m^2}}{f_{m,1}(\frac{z}{m})^{n^2}} =f_{n,m}(z), 
 \]
which means that the function $f_{n,m}(z)$ is periodic with respect to the lattice $m\Omega$. Hence by \cite[Chap.~9 Theorem 9.3]{Washington} $f_{n,m}(z)$ belongs to the field $\C(\wp(\frac{z}{m}), \wp'(\frac{z}{m}))$. From the Taylor expansion of the $\sigma$ function at the origin (see \cite[\S 10.5 page 391]{MOS} or \cite[\S 10 B Lemma 10.12]{Cox}), we deduce that the coefficients of the function $f_{n,m}(z)$ belong to $\Q\big(g_2,g_3)$:
 \[
 f_{n,m}(z) \in \Q\Big(g_2,g_3,\wp\Big(\frac{z}{m}\Big), \wp'\Big(\frac{z}{m}\Big)\Big).
 \] 
 We have $f_{0,m}(z)=0$. Assume now $n\not=0$. The logarithmic derivative of \eqref{def:f_{n,m}} gives 
 \begin{equation}\label{equation:zeta}
 \zeta\left(\frac{nz}{m}\right) = \frac{n}{m} \zeta(z) + \frac{f'_{n,m}(z)}{mn f_{n,m}(z)}\cdotp
\end{equation}
This formula shows that the function $ \frac{f'_{n,m}(z)}{mn f_{n,m}(z)}$ depends only on $n/m$. 
 
If the function $\frac{f'_{n,m}(z)}{ f_{n,m}(z)}$ is constant and $n\not=0$, then from equation \eqref{equation:zeta} we deduce that the poles of $\zeta\left(\frac{nz}{m}\right) $ are the same as the poles of $ \zeta(z) $, hence $(n/m)\Omega=\Omega$, which means that $n/m$ is an automorphism of $\cE$, and since $n/m$ is a rational number we have $n/m=\pm 1$. In particular $ f_{n,m}(z)$ is constant if and only if $n=\pm m$:
 \[
f_{m,m} (z)=1,\quad f_{-m,m} (z)=(-1)^m
 \]
for $m>0$. 

One more derivative yields
 \[ 
 \wp \left(\frac{nz}{m}\right) = \wp(z) - \frac{f''_{n,m}(z)f_{n,m}(z)-f'_{n,m}(z)^2}{n^2 f_{n,m}(z)^2} \cdotp
 \] 
\null\hfill$\square$
\end{proof}
 
 \begin{lemma}\label{Lemma:division}
 Let $m\geqslant 1$.
 \begin{enumerate}
\item
The field $\Q(g_2,g_3,\wp(z))$ is a finite extension of $\Q(g_2,g_3,\wp(mz))$.
\item
Let $p\in\C\smallsetminus\Omega$. Then $\wp(p/m)$ and $\zeta(p/m) -\zeta(p)/m$ are algebraic over the field $\Q(g_2,g_3,\wp(p))$.
\item
Let $\omega\in\Omega$. Assume $\omega/m\not\in\Omega$. Then $\wp(\omega/m)$ and $\zeta(\omega/m)-\eta(\omega)/m$ are algebraic over the field $\Q(g_2,g_3)$.
\end{enumerate}
 \end{lemma}
 
 \begin{proof}
 Item 1 follows from the proof of \cite[Th.~20.9 p.253]{M}. See also \cite[I.3]{F}. 
 According to \cite[Lemma 6.1]{M75} and \cite[Chap.~II]{LangECDA}, the degree of the extension $\Q(g_2,g_3,\wp(z),\wp'(z))$ over $\Q(g_2,g_3,\wp(mz),\wp'(mz))$ is $m^2$. 
 
 Item 2 deals with the division points. Since $m\Omega\subset \Omega$, the assumption $p\not\in\Omega$ implies $p/m\not\in\Omega$. From item 1 it follows that there is a nonzero polynomial $A\in\Q(g_2,g_3)[X,Y]$ such that the meromorphic function $A(\wp(z),\wp(z/m))$ is $0$. In the polynomial ring $\Q(g_2,g_3,\wp(p))[X,Y]$ we divide $A(X,Y)$ by $X-\wp(p)$. Let $\tilde{A}$ be the quotient:
 \[
 A(X,Y)=(X-\wp(p))^h \tilde{A}(X,Y),
 \]
 where $h\geqslant 0$ and the polynomial $\tilde{A}(\wp(p),Y)\in \Q(g_2,g_3,\wp(p))[Y]$ is not $0$. From $\tilde{A}(\wp(p),\wp(p/m))=0$ we deduce the result for $\wp(p/m)$. 
 
 The multiplication formula for $\zeta$ implies that for $p\in\C\smallsetminus\Omega$ and $m\geqslant 1$, the number $\zeta(p/m) -\zeta(p)/m$ belongs to $\Q(g_2,g_3,\wp(p/m), \wp'(p/m))$.
 
 Item 3 deals with the torsion points. Let $2^a$ with $a\geqslant 0$ be the highest power of $2$ such that $\omega':=\omega/2^a\in\Omega$. Using item 2 with $p=\omega'/2$, $m'=m/2$, so that $p/m'=\omega'/m$, we deduce that $\wp(p/m')$ and $\zeta(p/m')-\zeta(p)/m'$ are algebraic over the field $\Q(g_2,g_3)$.
 Since
\[
\frac p {m'}=\frac{\omega'} m=\frac \omega {2^a m},\quad\text{and}\quad
\zeta(p)/m'=\eta(\omega')/2m'=\eta(\omega)/2^a m,
 \] 
 the numbers
 \[
 \wp(\omega'/m)=\wp(\omega /2^a m)\quad\text{and}\quad
 \zeta(\omega'/m)-\eta(\omega')/m=\zeta(\omega /2^a m)-\eta(\omega)/2^am
 \]
 are algebraic over the field $\Q(g_2,g_3)$. The formulae of multiplication by $2^a$ (Table 2) complete our proof. 
 \null\hfill$\square$
\end{proof}

\begin{remark}\rm 
Several of our statements have the following shape:
\begin{quote}
\em
Let $p_1,\dots,p_n\in\C\smallsetminus\Omega$. For given complex numbers $z_1,\dots,z_m$ and a positive integer $\vartheta$, under suitable assumptions, the transcendence degree of the field 
\[
\Q\bigl(g_2,g_3,p_1,\dots,p_n,\wp(p_1),\dots,\wp(p_n),\zeta(p_1),\dots,\zeta(p_n),z_1,\dots,z_m\bigr)
\] 
is at least $\vartheta$.
\end{quote}

We point out that the assumption $p_i\not\in\Omega$ could be dropped if we were replacing this statement with the following equivalent one, where $\tilde{\cE}$ denotes the nontrivial extension of $\cE$ by $\G_a$, the exponential of which is (see \cite[Exercise 20.102]{M}):
\[
\begin{matrix}
 \exp_{\tilde{\cE}}:&\Lie(\tilde{\cE}) _{\C}= \C^2 &\longrightarrow& \tilde{\cE} (\C) \subset \PP^{4}(\C)
 \hfill
\\
 &(w,z) & \longmapsto & 
\sigma(z)^3 [\wp(z) : \wp'(z) :  1:
w+\zeta(z):(w+\zeta(z))\wp'(z)+2\wp(z)^2].
 \end{matrix}
\] 

\begin{quote}
\em
Let $p_1,\dots,p_n\in\C$. Let $K_0$ be a field of definition of $\tilde\cE$ containing the coordinates of $\exp_{\tilde\cE}(p_i)$ for $1\leqslant i\leqslant n$. Then for given complex numbers $z_1,\dots,z_m$ and a positive integer $\vartheta$, the transcendence degree (over $\Q$, as always) of the field $K_0\bigl(z_1,\dots,z_m\bigr)$ is at least $\vartheta$.
\end{quote}
 
This alternative formulation would avoid the need to consider separately the case where some $p_i$ belongs to $\Omega$.

For instance items 2 and 3 of Lemma \ref{Lemma:division} mean that for $(w,p)\in\C^2$, if $K$ is a field of definition of $\tilde\cE$ containing the coordinates of $\exp_{\tilde\cE}(w,p)$, then $\exp_{\tilde\cE}(w/m,p/m)$ for $m\geqslant 1$ is defined over the algebraic closure of $K$. In other terms for $p\not\in\Omega$ if $w+\zeta(p)$ and $\wp(p)$ are algebraic over a field $K$ containing $\Q(g_2,g_3)$, then $w/m+\zeta(p/m)$ and $\wp(p/m)$ also; while for $\omega\in\Omega$, if $w+\eta(\omega)$ is algebraic over a field $K$ containing $\Q(g_2,g_3)$, then $w/m+\zeta(\omega/m)$ also. 
 \end{remark}
 
 \subsection{Including the CM case} \label{S:CM}

Assume that the elliptic curve $\cE$ has complex multiplication. Let $\tau=\frac{\omega_2}{\omega_1}$ be the quotient of a pair of fundamental periods of $\wp$. Then $k$ is the imaginary quadratic extension $k=\Q(\tau)$ of $\Q$ and $\tau$ is a root of a polynomial 
\[ 
A+BX+CX^2 \in \Z[X],
\] 
where $A,B,C$ are relatively prime integers with $C>0$. Hence $C \tau \Omega \subseteq \Omega$.

 According to \cite[Chap. III, \S3.2, Lemma 3.1] {M75}\footnote{At the beginning of the book \cite{M75}, the author assumes the invariants $g_2,g_3$ algebraic but his Lemma 3.1 remains true even without this hypothesis. \cite[Appendix B, Th. 8]{BK} does not assume that $g_2$ and $g_3$ are algebraic.} 
and \cite[Appendix B, Th. 8]{BK}, there are two independent linear relations between the periods $\omega_1, \omega_2,\eta_1,\eta_2$, namely
\[
\omega_2-\tau\omega_1 =0
\]
and 
\begin{equation}\label{Equation:kappa}
A\eta_1- C\tau\eta_2-\kappa\omega_2=0
\end{equation}
where $\kappa$ is algebraic over the field $\Q(g_2,g_3)$.

By \cite[Lemma 6.2]{M75} and \cite[Proposition 5]{Reyssat82}, there exist two non constant polynomials $P(X),Q(X)$ in $ \Q(g_2,g_3,\tau)[X]$ such that the following three equa\-li\-ties hold: 

\medskip
\medskip

\begin{tabular}{ |p{15cm}| }
\hline
\multicolumn{1}{|c|}{} \\
 \multicolumn{1}{|c|}{\textbf{ Multiplication by $\tau=\frac{\omega_2}{\omega_1}$ formulae }} \\
\multicolumn{1}{|c|}{(only in the case of complex multiplication)} \\
\multicolumn{1}{|c|}{} \\
\hline \rule[-4mm]{0cm}{1cm}
$\wp (C\tau z) = \frac{P(\wp(z))}{Q(\wp(z))} \in \Q\big(g_2,g_3,\tau, \wp(z)\big) $ \\
\hline \rule[-4mm]{0cm}{1cm}
$C \tau\zeta(C\tau z) = AC \zeta(z)- \kappa C\tau z + \frac{\wp'(z) Q'(\wp(z))}{2Q(\wp(z))} $ \\ 
\hline \rule[-4mm]{0cm}{1cm}
$ \sigma(C\tau z)^2 = (C\tau)^2 \sigma(z)^{2AC} \; \rme^{-\kappa  C\tau z^2} \; Q(\wp(z))$ \\ 
\hline 
 \end{tabular}
 
 \begin{center}
{\bf Table 3}
 \end{center}

 \subsection{Auxiliary results} % \label{S:AuxiliaryResults}

Here is the main result of this section. 

 \begin{proposition}\label{Proposition:MultiplicationFormula}
 Let $\alpha$ be a nonzero element of $k$. Write $\alpha=r_1+r_2\tau$ where $r_1$ and $r_2$ are two rational numbers, not both zero. Let $m\in\Z$ be the least positive integer such that $mr_1$ and $mr_2/C$ are integers. Then
\\
(1)
The function $\wp\bigl( \alpha z \bigr)$ belongs to $k(g_2,g_3,\wp(z/m))$. 
\\
(2) 
The function $\Xi_{r_1,r_2}$ defined by 
\[
\zeta\bigl( \alpha z \bigr)=\left( r_1+\frac A {C\tau} r_2\right)\zeta(z) - \frac {\kappa r_2} {C} z +\Xi_{r_1,r_2}(z)
\] 
belongs to $k(g_2,g_3,\wp(z/m),\wp'(z/m))$. 
 \end{proposition} 

\begin{proof}
Write $r_1=n_1/m$ and $r_2=Cn_2/m$, so that $n_1$ and $n_2$ are integers and 
\[
m\alpha=n_1+n_2 C\tau\in \End(\cE)\smallsetminus\{0\}.
 \] 
Item (1) follows from the fact that $\wp(\alpha z)$ is an even elliptic function for the lattice $m\Omega$ (compare with \cite[Lemma 6.3]{M75}).

Consider item (2). 
In case $r_2=0$ and $r_1\not=0$ we apply the second row of Table 2: 
\[
\Xi_{r_1,0}(z)=\frac{f'_{n_1,m}(z)}{mn_1 f_{n_1,m}(z)}\cdotp
 \]
As pointed out above \eqref{equation:zeta}, the right hand side depends only on $n_1/m$:
 \[
\zeta(r_1z)=r_1\zeta(z)+\Xi_{r_1,0}(z).
 \]
In case $r_1=0$ and $r_2=C$, we have $n_2=m=1$ and we apply the second row of Table 3: 
\[
\Xi_{0,C}(z)=\frac{\wp'(z) Q'(\wp(z))}{2C\tau Q(\wp(z))} \cdotp
 \]
Using these two cases, we deduce the result when $r_1=0$ and $r_2\not=0$ with
\[
\Xi_{0,r_2}(z)= \frac{r_2} C \Xi_{0,C}(z)+\Xi_{r_2/C,0}(C\tau z).
\] 
Finally when $r_1r_2\not=0$ we apply the second row of Table 1: 
\[
\Xi_{r_1,r_2}=\Xi_{r_1,0}(z)+\Xi_{0,r_2}(z)+
\frac{1}{2} \frac{\wp'(r_1z)-\wp'(r_2\tau z)}{\wp(r_1z)-\wp(r_2\tau z)}\cdotp
 \]
\null\hfill$\square$
\end{proof}

We will also need the following two consequences of Proposition \ref{Proposition:MultiplicationFormula}:

\begin{corollary}\label{CorollaryMultiplicationFormula}
Let $\alpha\in k$, $\alpha\not=0$. Then there exist two elements $\beta_1$ and $\beta_2$ in $k$ such that the function 
\begin{equation}\label{Equation:Xialpha}
\Xi_\alpha(z)=\zeta(\alpha z)-\beta_1\zeta(z)-\beta_2z
\end{equation}
is algebraic over the field $\Q(g_2,g_3,\wp(z))$. Moreover this function $\Xi_\alpha(z)$ is constant if and only if either $\alpha=\pm 1$, or $\alpha^4=1$ and $g_3=0$, or $\alpha^6=1$ and $g_2=0$, in which cases it is $0$. 
\end{corollary}

\begin{proof}
The first part of this statement follows immediately from Proposition \ref{Proposition:MultiplicationFormula}. It remains to elucidate the cases where the function $\Xi_\alpha(z)$ is constant. If either $\alpha=\pm 1$, or $\alpha^4=1$ and $g_3=0$, or $\alpha^6=1$ and $g_2=0$, then from Remark \ref{Remark:specialcases} we deduce $\beta_2=0$ and $\Xi_\alpha(z)=0$. Conversely, if $\Xi_\alpha(z)$ is constant, then the function $\zeta(\alpha z)-\beta_1\zeta(z)$ has no pole, hence $\beta_1\not=0$ and $\alpha\Omega=\Omega$. In this case $\alpha$ is an automorphism of the elliptic curve $\cE$, hence a root of unity in the number field $k$.
\hfill
\null\hfill$\square$
\end{proof}

\begin{corollary}\label{Corollary:linearrelations}
Let $\omega\in\Omega$, let $p_1,\dots,p_{n+1}$  in $\C\smallsetminus\Omega$  and let $b_0,b_1,\dots,b_n$   in $k$. Assume 
 \[
p_{n+1}=b_0\omega +b_1p_1+\cdots+b_n p_n.
\]
Then $\wp(p_{n+1})$ and $\zeta(p_{n+1})$ are algebraic over the field 
\[
\Q(g_2,g_3,\omega_1,\omega_2,\eta_1,\eta_2,p_1,\dots,p_n,\wp(p_1),\dots,\wp(p_n),\zeta(p_1),\dots,\zeta(p_n)).
\]
\end{corollary}

\begin{proof}
Using the addition formulae (Table 1), one deduces Corollary \ref{Corollary:linearrelations}  by induction from Lemma \ref{Lemma:division} and Proposition \ref{Proposition:MultiplicationFormula}.
\null\hfill$\square$
\end{proof}

%-----------------------------------------
\section{Geometric origin of the Split Semi--Elliptic Conjecture \ref{Conjecture:SplitSemiElliptic}}\label{S:Elliptico--toric}

In \cite[\S 1]{B02} the first author states the following conjecture, which she proves to be equivalent to the Grothendieck-Andr\'{e} generalized period Conjecture applied to the 1-motive 
\[
M=[u:\Z \rightarrow \G_m^s \times \cE^n ], \; u(1)=(\rme^{t_1}, \dots, \rme^{t_s} , P_1, \dots, P_n ) \in (\G_m^s\times \cE^n )(\C),
\] 
with $ P_i=[\wp(p_i): \wp'(p_i):1] $ for $i=1, \dots,n:$

 \begin{conjecture} [Elliptico--Toric Conjecture] 
\label{Conjecture:elliptico--toric}
Let $s\geqslant 0$ and $n\geqslant 0$ be two integers,
$\rme^{t_1}, \dots, \rme^{t_s}$ be points in $\G_m(\C)$ and $P_1, \dots, P_n$ be points in $\cE(\C)$. Let $K$ be the field defined by \eqref{Equation:K}. Then
\begin{equation}\label{Equation:trandeg} \mathrm{tran.deg} \; K ( \omega_1, \omega_2,\eta_1,\eta_2) \geqslant \dim_\Q <t_\ell>_\ell+ \frac{4}{[k:\Q]} +2 \dim_{k}<p_i>_i, 
 \end{equation}
 where$<t_\ell>_\ell$ is the sub $\Q $--vector space of $\C / 2\pi\rmi \Q$ generated by the classes of $t_1, \dots, t_s $ modulo $ 2\pi\rmi \Q$ 
 and $<p_i>_i$ is the sub $k$-vector space of $\C/(\Omega\otimes_\Z\Q)$ generated by the classes of $p_1, \dots, p_n$ modulo $\Omega\otimes_\Z\Q$. \\
\end{conjecture} 

Notice  that if the numbers $\rme^{t_\ell}, g_2,g_3,\wp(p_i)$ for $\ell=1,\dots,s$ and $i=1, \dots,n$ are all algebraic,  then, by Lemma \ref{ChoiceBase1} below,
\[
\mathrm{tran.deg} \; K ( \omega_1, \omega_2,\eta_1,\eta_2) \leqslant \dim_\Q <t_\ell>_\ell+ \frac{4}{[k:\Q]} +2 \dim_{k}<p_i>_i.
\]

\goodbreak
%\eqref{Equation:trandeg} becomes an equality.
\medskip

\begin{remark}\label{Remarque:GGPCEllipticCase}\rm
 \par\noindent
 \begin{enumerate}
\item 
Legendre relation \eqref{Equation:Legendre} is a polynomial relation between the periods $\omega_1, \omega_2,\eta_1,\eta_2$ of the elliptic curve $\cE$ and the period $2\pi\rmi$ of the multiplicative group $\G_m$ (see for instance \cite[Example 5.1]{BP}):
\[
 \Q (2\pi\rmi, \omega_1, \omega_2,\eta_1,\eta_2) = \Q (\omega_1, \omega_2,\eta_1,\eta_2).
 \]

\item Assume the elliptic curve $\cE$ has complex multiplication, i.e. $k=\Q(\tau)$. 
According to \eqref{Equation:kappa}, the numbers $\omega_2$ and $\eta_2$ are algebraic over the field $\Q (g_2,g_3,\omega_1,\eta_1)$. 
 
\item For $s=n=0$, Conjecture \ref{Conjecture:elliptico--toric} applied only to the elliptic curve $\cE$ reads
\begin{equation*}
\mathrm{tran.deg} \, \Q (g_2,g_3,\omega_1, \omega_2,\eta_1,\eta_2) \geqslant \frac{4}{[k:\Q]} =
\begin{cases}
 2, \mathrm{ if}\; \cE \;\mathrm{has\; complex\; multiplication}\\
 4, \mathrm{ otherwise}.
\end{cases} 
\end{equation*}
If $g_2$ and $g_3$ are algebraic, then unconditionally we have, by \eqref{Equation:kappa},   
\[
\mathrm{tran.deg} \, \Q (\omega_1, \omega_2,\eta_1,\eta_2) \leqslant \frac{4}{[k:\Q]}\cdotp
 \]
In the CM case, according to Chudnovsky's Theorem \cite[Chap. 7]{Ch}, 
\[
 \mathrm{tran.deg} \, \Q (\omega_1, \omega_2,\eta_1,\eta_2) =2.
\]
 
\item For $n=0$, Conjecture \ref{Conjecture:elliptico--toric} is the following statement, which implies Schanuel's Conjecture \ref{SchanuelConjecture}:
\begin{quote}
{\em Let $t_1,\dots,t_s$ be  complex numbers such that $2\pi\rmi,t_1,\dots,t_s$  are $\Q$--linearly independent. Let $\cE$ be an elliptic curve with algebraic invariants $g_2,g_3$ and field of endomorphisms $k$. 
Then at least $s+\frac 4 {[k:\Q]}$ of the $2s+4$ numbers 
\[
t_1,\dots,t_s, \rme^{t_1},\dots, \rme^{t_s}, \omega_1,\omega_2,\eta_1,\eta_2
\] 
are algebraically independent.}
\end{quote}
 \end{enumerate}
\end{remark}

The main goal of the present section is to prove: 

\begin{theorem}\label{G<=>SplitSchanuel}
 The Split Semi--Elliptic Conjecture \ref{Conjecture:SplitSemiElliptic} is equivalent to the Elliptico--Toric Conjecture \ref{Conjecture:elliptico--toric}. 
\end{theorem}

We introduce assumptions on $t_1, \dots,t_s$, $p_1,\dots,p_n$. For $t_1, \dots,t_s$, we have the choice between the condition 
\par
$({\mathrm {C}}_t$): {\em $t_1, \dots,t_s$ are $\Q$--linearly independent}
\\
and the stronger condition 
\par
(${\mathrm {C}}^\star_t$): {\em$2\pi\rmi, t_1, \dots,t_s$ are $\Q$--linearly independent.}

\noindent
The condition (${\mathrm {C}}^\star_t$) is equivalent to each of the following ones
\par
(a) {\em 
The classes modulo $2\pi\rmi\Q$ of $t_1, \dots,t_s$ are $\Q$--linearly independent.}
\par
(b)
{\em $t_1, \dots,t_s$ are $\Q$--linearly independent and $(\Q t_1+\cdots+\Q t_s)\cap (2\pi\rmi\Q)=\{0\}$}. 
\par
(c)
{\em The points $\rme^{t_1},\dots,\rme^{t_s}$ are multiplicatively independent in $\C^\times$. }

\medskip

\noindent
In the same way, for $p_1,\dots,p_n$, we have the choice between the condition 
\par
(${\mathrm {C}}_p$): {\em $p_1, \dots,p_n$ are $k$--linearly independent}
\\
and the stronger condition 

\par
(${\mathrm {C}}^\star_p$): 
{\em
$\omega_1,\omega_2, p_1,\dots,p_n$ are $\Q$--linearly independent in the non--CM case,
$\omega_1,p_1,\dots,p_n$ are $k$--linearly independent in the CM case.
}

\noindent
The condition (${\mathrm {C}}^\star_p$) is equivalent to each of the following ones
\par
(a)
{\em The classes of $p_1,\dots,p_n$ modulo $\Omega\otimes_\Z \Q=\Q\omega_1+\Q\omega_2$ are $k$--linearly independent. }
\par 
(b)
{\em $p_1,\dots,p_n$ are $k$--linearly independent and 
$(kp_1+\cdots+kp_n) \cap \Omega =\{0\}$}.
\par
(c)
{\em Let $P_i=\exp_\cE(p_i)\in \cE(\C)$ ($i=1,\dots,n$). Then the points $P_1,\dots,P_n$ are $k$--linearly independent.
}

The assumptions of Conjecture \ref{Conjecture:SplitSemiElliptic} are {\rm{(${\mathrm {C}}_t$)}} and {\rm{(${\mathrm {C}}_p$)}}.
Lemma \ref{ChoiceBase1} below shows that the Elliptico--Toric Conjecture \ref{Conjecture:elliptico--toric} is equivalent to the following statement, where the assumptions are {\rm{(${\mathrm {C}}^\star_t$)}} and {\rm{(${\mathrm {C}}^\star_p$)}}. 

\begin{conjecture}\label{Conjecture:elliptico-toricBis}
Assume {\rm{(${\mathrm {C}}^\star_t$)}} and {\rm{(${\mathrm {C}}^\star_p$)}}. 
Then the field $K(\omega_1,\omega_2,\eta_1,\eta_2)$ has transcendence degree at least 
\[
s+\frac 4{[k:\Q]}+2n.
\]
\end{conjecture}

\begin{lemma} \label{ChoiceBase1}
 Consider $s$ complex numbers $ t_1, \dots, t_s $ and $n$ complex numbers $ p_1, \dots, p_n $ in $ \C \smallsetminus \Omega$. Let $t'_1, \dots ,t'_{s'}$ be a $\Q$--basis of the sub $\Q $--vector space $<t_\ell>_\ell$ of $\C / 2\pi\rmi \Q$ generated by the classes of $t_1, \dots, t_s $, and let $p'_1, \dots ,p'_{n'}$ be a $k$--basis of the sub $k $--vector space $<p_i>_i$ of $\C/(\Omega\otimes_\Z\Q)$ generated by the classes of $p_1, \dots, p_n. $ Then the field 
 \[
\Q \big(t_\ell, \rme^{t_\ell}, g_2,g_3,\omega_1, \omega_2,\eta_1,\eta_2, p_i, \wp(p_i),\zeta(p_i) \big)_{\audessus{\ell=1,\dots,s}{i=1, \dots,n}} 
\]
is a finite extension of the field
 \[
 \Q \big(t_\ell', \rme^{t_\ell'},g_2,g_3, \omega_1, \omega_2,\eta_1,\eta_2, p_i', \wp(p_i'), \zeta(p_i') \big)_{\audessus{ \ell=1,\dots,s' }{i=1, \dots,n' }}.
 \]
 Therefore the two fields have the same algebraic closure, and the same transcendence degree. 
\end{lemma}

\begin{proof} [Proof of Lemma \ref{ChoiceBase1}]
For ease of notation, we assume, as we may without loss of generality, $t'_\ell=t_\ell$ for $1\leqslant \ell\leqslant s'$ and $u_i=u'_i$ for $1\leqslant i\leqslant n'$.

If $s>s'$, then there is a linear relation 
\[
t_s=a_0 2\pi\rmi+a_1t_1+\cdots+a_{s'}t_{s'}
\]
with $a_\ell \in\Q$ ($\ell=0,1,\dots,s'$). Hence $\rme^{t_s}$ is algebraic over the field $\Q(\rme^{t_1},\dots,\rme^{t_{s'}})$. Recall Legendre relation \eqref{Equation:Legendre}: $2\pi\rmi$ belongs to $\Q(\omega_1,\omega_2,\eta_1,\eta_2)$.

If $n>n'$, then there is a linear relation 
\[
p_n=b_0\omega +b_1p_1+\cdots+b_{n'}p_{n'}
\]
with $b_i \in k$ ($i=0,1,\dots,n'$) and $\omega\in\Omega$. Corollary \ref{Corollary:linearrelations} implies that $\wp(p_n)$ and $\zeta(p_n)$ are algebraic over the field 
$\Q(g_2,g_3,p_1,\dots,p_{n'}, \wp(p_1),\dots,\wp(p_{n'}),\zeta(p_1),\dots,\zeta(p_{n'}))$.
\null \hfill $\square$
\end{proof}

 \begin{proof}[Proof of Theorem \ref{G<=>SplitSchanuel}]
 It remains to prove that Conjectures \ref{Conjecture:SplitSemiElliptic} and \ref{Conjecture:elliptico-toricBis} are equivalent. 
 
 \medskip
 \noindent
 {\em Conjecture \ref{Conjecture:SplitSemiElliptic} $\Longrightarrow$ Conjecture \ref{Conjecture:elliptico-toricBis}}
 
 We will use only the exceptional case in Conjecture \ref{Conjecture:SplitSemiElliptic} where the transcendence degree is at least $s+2n-1$. 
 
Let $(t_1,\dots,t_s)$ satisfy {\rm{(${\mathrm {C}}^\star_t$)}} and $(p_1,\dots,p_n)$ satisfy {\rm{(${\mathrm {C}}^\star_p$)}}. 

Assume the elliptic curve $\cE$ is not a CM curve. We deduce from Conjecture \ref{Conjecture:SplitSemiElliptic} applied to $2\pi\rmi,t_1,\dots,t_s$ and $\omega_1/2,\omega_2/2,p_1,\dots,p_n$, with $s,n$ replaced by $s',n'$, where $s'=s+1$, $n'=n+2$, that the field $K(\omega_1,\omega_2,\eta_1,\eta_2)$ has transcendence degree at least $s'+2n'-1=s+2n+4$.

Assume the elliptic curve $\cE$ is a CM curve. We deduce from Conjecture \ref{Conjecture:SplitSemiElliptic} applied to $2\pi\rmi,t_1,\dots,t_s$ and $\omega_1/2,p_1,\dots,p_n$, with $s,n$ replaced by $s',n'$, where $s'=s+1$, $n'=n+1$, that the field $K(\omega_1,\omega_2,\eta_1,\eta_2)$ has transcendence degree at least $s'+2n'-1=s+2n+2$.
 
 \medskip
 \noindent
 {\em Conjecture \ref{Conjecture:elliptico-toricBis} $\Longrightarrow$ Conjecture \ref{Conjecture:SplitSemiElliptic}}
 
 Let $(t_1,\dots,t_s)$ satisfy {\rm{(${\mathrm {C}}_t$)}} and $(p_1,\dots,p_n)$ satisfy {\rm{(${\mathrm {C}}_p$)}}. 
 We consider two cases for $(t_1,\dots,t_s)$,  denoted $(T)$ and ($T^\star$): 
 \[
 \dim_\Q \bigl((2\pi \rmi \Q )\cap ( \Q t_1+\cdots +\Q t_s) \bigr)=
 \begin{cases}
 1&\hbox{called {\em Case $(T)$}}
 \\
 0&\hbox{called {\em Case $(T^\star)$}}
 \end{cases}
 \]  
  We consider three cases for $(p_1,\dots,p_n)$ denoted ($P$), ($\tilde{P}$) and ($P^\star$): 
  \[
 \dim_k \bigl( (\Omega\otimes_\Z\Q ) \cap ( k p_1+\cdots+k p_n)   \bigr)=
 \begin{cases}
 2&\hbox{called {\em Case $(P)$}}
 \\ 
 1&\hbox{called {\em Case $(\tilde{P})$}}
 \\
 0&\hbox{called {\em Case $(P^\star)$}}
 \end{cases}
 \]  
 When  $s=0$ we are in the case $(T^\star)$ while when  $n=0$ we are in the case  $(P^\star)$. 

 We will need to consider several cases. The following reductions will introduce no loss of generality.
 
 \noindent
 $\bullet$
 $(T)$: \quad
 $2\pi\rmi\in \Q t_1+\cdots +\Q t_s$. 
 \par
 In this case we will assume $t_1=2\pi\rmi$, then $(t_2,\dots,t_s)$ satisfy {\rm{(${\mathrm {C}}^\star_t$)}}. We will apply Conjecture \ref{Conjecture:elliptico-toricBis} with  $s$ replaced by $s'=s-1$. Notice that $K$ contains $2\pi \rmi$.

 \noindent
 $\bullet$
($T^\star$): \quad 
 $(t_1,\dots,t_s)$ satisfy {\rm{(${\mathrm {C}}^\star_t$)}}. 
 \par
 In this case we will apply Conjecture \ref{Conjecture:elliptico-toricBis} with $s'=s$.

 \noindent
 $\bullet$
($P$): \quad 
$\Omega\subset k p_1+\cdots+k p_n$. 
 \par
 In the non--CM case we will assume $p_1=\omega_1/2$, $p_2=\omega_2/2$, so that $(p_3,\dots,p_n)$ satisfy {\rm{(${\mathrm {C}}^\star_p$)}} and we will apply Conjecture \ref{Conjecture:elliptico-toricBis} with $n$ replaced by $n'=n-2$. 
 \par
 In the CM case we will assume $p_1=\omega_1/2$, so that  $(p_2,\dots,p_n)$ will satisfy {\rm{(${\mathrm {C}}^\star_p$)}} and we will apply Conjecture \ref{Conjecture:elliptico-toricBis} with $n$ replaced by $n'=n-1$. 
 
In both cases non--CM or CM the numbers $\omega_1,\omega_2,\eta_1,\eta_2$ are algebraic over the field $K$.
 
 \noindent
 $\bullet$
($\tilde{P}$): \quad 
$\Omega\not\subset k p_1+\cdots+k p_n$ and $\Omega\cap(k p_1+\cdots+k p_n)\not=0$. 
 \par
 In this case $\cE$ is a non--CM elliptic curve and we will assume $p_1=\omega_1/2$, so that $(p_2,\dots, p_n)$ will satisfy {\rm{(${\mathrm {C}}^\star_p$)}} and we will apply Conjecture \ref{Conjecture:elliptico-toricBis} with $n$ replaced by $n'=n-1$. 
Notice that $K$ contains $\omega_1,\eta_1$.
 
 \noindent
 $\bullet$
 ($P^\star$): 
 \quad 
$(p_1,\dots,p_n)$ satisfy {\rm{(${\mathrm {C}}^\star_p$)}}. 
 \par
 In this case we will apply Conjecture \ref{Conjecture:elliptico-toricBis} with $n'=n$.

 \medskip 
The following remark will be useful: 
 Let $S$ be a finite subset of $\C$. Then 
\begin{equation}\label{Equation:remarque}
\mathrm{tran.deg}_\Q \, K = \mathrm{tran.deg}_\Q \, K(S) -\mathrm{tran.deg}_K \, K(S).
\end{equation}
In particular if we know that the transcendence degree of $K(S)$ over $\Q$ is at least $\tau+\sigma$ and that the transcendence degree of $K(S)$ over $K$ is at most $\sigma$, then  the transcendence degree of $K$ over $\Q$ is at least $\tau$ (in other terms {\em we remove $S$}).

\medskip

 We now consider each of the six cases.

 \noindent
 $\diamond$
 ($TP$): 
 \par
This is the exceptional case in Conjecture \ref{Conjecture:SplitSemiElliptic} where the conclusion is that the transcendence degree is at least $s+2n-1$. 

We have $s'=s-1$, $n'=n-\frac 2 {[k:\Q]}$, $s'+\frac 4 {[k:\Q]}+2n'=s+2n-1$. The field $K$ contains $\omega_1,\omega_2,\eta_1,\eta_2$.
 
 \noindent
 $\diamond$
 ($T^\star P$): 
 \par
 We have $s'=s$, $n'=n-\frac 2 {[k:\Q]}$, $s'+\frac 4 {[k:\Q]}+2n'=s+2n$. The numbers $\omega_1,\omega_2,\eta_1,\eta_2$ are algebraic over the field $K$.

 \noindent
 $\diamond$
 ($T\tilde{P}$): 
 \par
 $s'=s-1$, $n'=n-1$, $k=\Q$, $s'+2n'+ 4 =s+2n+1$. According to Conjecture \ref{Conjecture:elliptico-toricBis} the field $K(\eta_2)$ has transcendence degree at least $s+2n+1$. From remark \eqref{Equation:remarque} with $S=\{\eta_2\}$ we deduce that $K$ has transcendence degree at least $s+2n$.
 
 \noindent
 $\diamond$
 ($T^\star \tilde{P}$): 
 \par
 $s'=s$, $n'=n-1$, $k=\Q$, $s'+2n'+ 4 =s+2n+2$. According to Conjecture \ref{Conjecture:elliptico-toricBis} the field $K(\omega_2,\eta_2)$ has transcendence degree at least $s+2n+2$. From  remark \eqref{Equation:remarque}  with $S=\{\omega_2,\eta_2\}$ we deduce that  $K$ has transcendence degree at least $s+2n$.

 \noindent
 $\diamond$
 ($TP^\star$): 
 \par
 $s'=s-1$, $n'=n$, $s'+\frac 4 {[k:\Q]}+2n' =s+\frac 4 {[k:\Q]}-1+2n$. According to Conjecture \ref{Conjecture:elliptico-toricBis} the field $K(\omega_1,\omega_2,\eta_2)$ has transcendence degree at least $s+\frac 4 {[k:\Q]}-1+2n$. Since $\Q(g_2,g_3,\omega_1,\omega_2,\eta_2)$ has transcendence degree at most $\frac 4 {[k:\Q]}-1$ over $\Q(g_2,g_3)$, we deduce from  remark \eqref{Equation:remarque}  with $S=\{\omega_1,\omega_2,\eta_2\}$  that $K$ has transcendence degree at least $s+2n$.
 
 \noindent
 $\diamond$
 ($T^\star P^\star$): 
 \par 
 $s'=s$, $n'=n$. 
 Conjecture \ref{Conjecture:elliptico-toricBis} implies that the field $K(\omega_1,\omega_2,\eta_1,\eta_2)$ has transcendance degree at least $s+\frac 4 {[k:\Q]}+2n$. Since $\Q(g_2,g_3,\omega_1,\omega_2,\eta_1,\eta_2)$ has transcendance degree at most $\frac 4 {[k:\Q]}$ over $\Q(g_2,g_3)$, we deduce from  remark \eqref{Equation:remarque}  with $S=\{\omega_1,\omega_2,\eta_2\}$  that $K$ has transcendance degree at least $s+2n$. 
\null\hfill$\square$
\end{proof}
 
%-----------------------------------------
\section{Some consequences of the Split Semi--Elliptic Conjecture \ref{Conjecture:SplitSemiElliptic}}\label{S:ConsequencesSplitElliptic} 

We give proofs of several consequences of Conjecture \ref{Conjecture:SplitSemiElliptic} stated in section \ref{S:The Split Semi--Elliptic Conjecture}

 We start with a proof that the Split Semi--Elliptic Logarithms Conjecture \ref{Conjecture:SplitSemiEllipticLogarithms} is a consequence of Conjecture \ref{Conjecture:SplitSemiElliptic}.

 \begin{proof}[Proof of Conjecture \ref{Conjecture:SplitSemiElliptic} $\Longrightarrow$ Conjecture \ref{Conjecture:SplitSemiEllipticLogarithms}]
 If either $2\pi\rmi\not\in \Q t_1+\cdots+\Q t_s$ or $\Omega\not\subset k p_1+\cdots+kp_n$, then Conjecture \ref{Conjecture:SplitSemiElliptic} implies that the transcendence degree of the field
 \[
\Q\bigl( \log\alpha_1,\dots,\log\alpha_s,\; p_1,\dots,p_n, \wp(p_{m+1)}, \dots,\wp(p_n), \zeta(p_1), \dots,\zeta(p_m)\bigr)
 \]
 is at least $s+2n$. If we remove the $n$ numbers $\zeta(p_1), \dots,\zeta(p_m), \wp(p_{m+ 1)}, \dots,\wp(p_n)$ (see remark \eqref{Equation:remarque}), the transcendence degree is at least $s+n$.

 Now we assume $2\pi\rmi\Q\subset \Q t_1+\cdots+\Q t_s$ and $\Omega\subset k p_1+\cdots+kp_n$. Without loss of generality we may assume $t_1=2\pi\rmi$ and $2p_1,2p_2\in\Omega$ in the non--CM case, $2p_1\in\Omega$ in the CM case.
 
 We start with the non--CM case. Since $\eta_1$ and $\eta_2$ are transcendental, we have $m\geqslant 3$. Recall Legendre relation \eqref{Equation:Legendre}: $\eta_2\in \Q(2\pi\rmi,\omega_1, \omega_2,\eta_1)$. Conjecture \ref{Conjecture:SplitSemiElliptic} implies that the transcendence degree of the field
 \[
\Q\bigl(2\pi\rmi, \log\alpha_2,\dots,\log\alpha_s,\; \omega_1,\omega_2,\eta_1, p_3,\dots,p_n, \wp(p_{m+1}), \dots,\wp(p_n), \zeta(p_3), \dots,\zeta(p_m)\bigr)
 \]
 is at least $s+2n-1$. If we remove the $n-1$ numbers $\eta_1, \wp(p_{m+ 1}), \dots,\wp(p_n), \zeta(p_3), \dots,\zeta(p_m) $, the transcendence degree of the remaining field is at least $s+n$.

 Finally consider the CM case. Since $\eta_1$ is transcendental, we have $m\geqslant 2$. From \eqref{Equation:Legendre} and \eqref{Equation:kappa} it follows that $\omega_2$, $\eta_1$ and $\eta_2$ are algebraic over the field $\Q(2\pi\rmi,\omega_1)$. Conjecture \ref{Conjecture:SplitSemiElliptic} implies that the transcendence degree of 
 \[
\Q\bigl(2\pi\rmi, \log\alpha_2,\dots,\log\alpha_s,\; \omega_1,\ p_2,\dots,p_n, \wp(p_{m+1}), \dots,\wp(p_n), \zeta(p_2), \dots,\zeta(p_m)\bigr)
 \]
 is at least $s+2n-1$. If we remove the $n-1$ numbers $ \wp(p_{m+ 1}), \dots,\wp(p_n), \zeta(p_2), \dots,\zeta(p_m) $, the transcendence degree of the remaining field is at least $s+n$.
\null\hfill$\square$
\end{proof}
 
Let us show now that Conjecture \ref{Conjecture:SplitSemiElliptic} implies a strong form of Schneider's Theorem \ref{Th:Schneider}. 
 
\begin{proposition}\label{G=>S}
Assume Conjecture \ref{Conjecture:SplitSemiElliptic}.
\begin{enumerate}
\item
Let $p\in\C \smallsetminus\Omega$. Then the transcendence degree of the field generated by the five numbers 
\[
 g_2 ,\; g_3 ,\; p ,\; \wp(p) ,\; \zeta(p)
\]
is at least $2$. 
 \item
 Let $t$ be a nonzero complex number and let $p\in\C \smallsetminus\Omega$. Then the transcendence degree of the field generated by the seven numbers 
\[
t,\; \rme^t,\; g_2,\;g_3,\; p,\; \wp(p),\; \zeta(p)
\]
is at least $3$, unless $\cE$ is a CM curve, $t\in 2\pi\rmi \Q$ and $p\in\Omega\otimes_\Z\Q$, in which case it is only $2$.
\item
Let $p_1$ and $p_2$ be two elements of $\C \smallsetminus\Omega$ such that $p_2/p_1\not\in k$. Then the transcendence degree of the field generated by the eight numbers 
\[
g_2,\;g_3,\; p_1,\; p_2,\; \wp(p_1),\; \wp(p_2),\; \zeta(p_1),\;\zeta(p_2)
\]
is at least $4$.
\end{enumerate} 
\end{proposition}

\begin{proof} [Proof of Proposition \ref{G=>S}]
We assume Conjecture \ref{Conjecture:SplitSemiElliptic}. Hence Conjecture \ref{Conjecture:EllipticSchanuel} is true. 
\\
(1) 
This is the case $n=1$ of Conjecture \ref{Conjecture:EllipticSchanuel}. 
\\
(2)
We apply Conjecture \ref{Conjecture:SplitSemiElliptic} with $s=n=1$. The transcendence degree is at least $s+2n=3$, unless $2\pi\rmi\in\Q t_1$ and $\Omega\subset kp_1$, in which case it is at least $s+2n-1=2$. 
\\
(3)
This is the case $n=2$ of Conjecture \ref{Conjecture:EllipticSchanuel}.
\null\hfill$\square$
\end{proof}

 \begin{proof} [Proof of Proposition \ref{G=>S} $\Longrightarrow$ Schneider's Theorem \ref{Th:Schneider}]
With the notations of Theorem \ref{Th:Schneider}, we  assume that $g_2 , g_3$ and $ \wp(p)$ are algebraic and  that Proposition \ref{G=>S} is true.
\\
(1) From part 1 of Proposition \ref{G=>S}, we deduce that  $p$ and $\zeta(p)$ are algebraically independent, hence $1$, $p$ and $\zeta(p)$ are linearly independent over the field of algebraic numbers. 
\\
(2)  In part 2 of Proposition \ref{G=>S}, the special case where $\cE$ is a CM curve, $t\in 2\pi\rmi \Q$ and $p\in\Omega\otimes_\Z\Q$ means that if $\omega$ is a nonzero period, then $\omega$ and $\pi$ are algebraically independent (recall Legendre relation \eqref{Equation:Legendre} and relation \eqref{Equation:kappa}). 
\\
Consider part 2 of Theorem \ref{Th:Schneider}. Set $t=\alpha p$. The transcendence of $\omega/\pi$ proves the desired result  when $\cE$ is a CM curve, $t\in 2\pi\rmi \Q$ and $p\in\Omega\otimes_\Z\Q$. Otherwise the transcendence degree of the field  
\[
\Q\bigl(\alpha,\; \rme^{\alpha p},\;   p,\;   \zeta(p)\bigr)
\]
is at least $3$, hence if $\alpha$ is algebraic then $ \rme^{\alpha p},\;   p,\;   \zeta(p)$  are algebraically independent, and therefore $ \rme^{\alpha p}$ is transcendental.
\\
(3) Assume that $\alpha$ is algebraic. Set $p_1=p$, $p_2=\alpha p_1$. By assumption the algebraic number $\alpha=p_2/p_1$ is not in $k$.
By part 3 of Proposition \ref{G=>S} the four numbers 
\[
 p,\; \wp(\alpha p),\; \zeta(p),\;\zeta( \alpha p)
\]
are algebraically independent and consequently $\wp(\alpha p)$ is transcendental. 
\null\hfill$\square$
\end{proof}

The $\zeta$--Conjecture \ref{Conjecture:zeta}. is also a consequence of Conjecture \ref{Conjecture:SplitSemiElliptic}, as shown by the next result.

\begin{proposition}\label{G=>W}
 Conjecture \ref{Conjecture:EllipticSchanuel} implies the $\zeta$--Conjecture \ref{Conjecture:zeta}.
\end{proposition}
 
The following auxiliary result \cite[Lemma 3.1]{K} will be useful:

\begin{lemma}
[Senthil Kumar]
\label{Lemma:SenthilKumar}
Let $\Omega$ be a lattice in $\C$ with invariants $g_2$, $g_3$. Let $K$ be a subfield of $\C$, let $f$  a nonconstant function in $K(g_2,g_3,\wp(z),\wp'(z))$ and $p\in\C\smallsetminus\Omega$; assume that $\wp(p)$ is transcendental over the field $K(g_2,g_3)$ and that $p$ is not a pole of $f$. Then $f(p)$ is transcendental over the field $K(g_2,g_3)$.
\end{lemma}

\begin{proof} [Proof of Proposition \ref{G=>W}]
We assume that $g_2$ and $g_3$ are algebraic.
Let $p\in\C \smallsetminus\Omega$. Assume that $\zeta(p$) is algebraic. According to Schneider's Theorem \ref{Th:Schneider}.1, $\wp(p)$ is transcendental. In particular (Lemma \ref{Lemma:division}.3) $P$ is not a torsion point: $p\not\in\Omega\otimes_\Z\Q$. Let $\alpha\in\C\smallsetminus\{0\}$.

Assume Conjecture \ref{Conjecture:EllipticSchanuel}.
\\
(1)
We first show that the assumption that $\zeta(p)$ is algebraic implies that for any nonzero algebraic number $\beta$, we have $\beta p\not\in\Omega$. 
Otherwise assume $\omega:=\beta p\in\Omega$. Since $p$ is not a torsion point and since $\beta p$ is a torsion point, we deduce $\beta\not\in k$. From Conjecture \ref{Conjecture:EllipticSchanuel} with $n=3$, $p_1=\omega_1/2$, $p_2=\omega_2/2$, $p_3=\omega/\beta$ in the non--CM case, $n=2$, $p_1=\omega_1/2$, $p_2=\omega/\beta$ in the CM case, we deduce that the two numbers $\wp(\omega/\beta)=\wp(p)$ and $\zeta(\omega/\beta)=\zeta(p)$ are algebraically independent over the field $\Q(\omega_1,\omega_2,\eta_1,\eta_2)$. This implies that $\zeta(p)$ is transcendental, which is a contradiction. 

In particular $\alpha p\not\in\Omega$.
\\
(2)
Assume $\alpha\not\in k$. From (1) it follows that $(k+k\alpha)p\cap\Omega=\{0\}$, hence $p$ and $\alpha p$ are $k$--linearly independent modulo $\Omega\otimes_\Z\Q$. Since $g_2$, $g_3$, $\alpha$ and $\zeta(p)$ are algebraic, item 3 of Proposition \ref{G=>S} with $p_1=p$ and $p_2=\alpha p$,
implies that the four numbers 
\[
p,\; \wp( p),\; \wp(\alpha p),\; \zeta( \alpha p)
\] 
are algebraically independent. In particular $\zeta(\alpha p)$ is transcendental. 
\\
(3) Assume $\alpha \in k$. Recall that $p\not\in\Omega\otimes_\Z\Q$.
Since $g_2$, $g_3$ and $\zeta(p)$ are algebraic, item 1 of Proposition \ref{G=>S} implies that the two numbers $p$ and $\wp( p)$ are algebraically independent. 
We apply Corollary \ref{CorollaryMultiplicationFormula}: from the hypotheses of Proposition \ref{G=>W}, we deduce that the elliptic function $\Xi_\alpha$ in \eqref{Equation:Xialpha} is not constant. By Lemma \ref{Lemma:SenthilKumar} the number $\Xi_\alpha(p)$ is transcendental. The two transcendental numbers $\Xi_\alpha(p)$  and $\wp(p)$ are algebraically dependent, hence the fields $\Q\bigl(\Xi_\alpha(p)\bigr)$ and $\Q(\wp(p)\bigr)$ have the same algebraic closure. 
From  \eqref{Equation:Xialpha} we deduce that the three fields $\Q\bigl(p,\zeta(p),\zeta(\alpha p)\bigr)$, $\Q\bigl(p,\zeta(p),\zeta(\alpha p),\Xi_\alpha(p)\bigr)$ and  $\Q\bigl(p,\zeta(p),\zeta(\alpha p),\wp(p)\bigr)$ have the same algebraic closure, hence have transcendence degree at least $2$. 
Since $\zeta(p)$ is algebraic, if follows that $p$ and $\zeta(\alpha p)$ are algebraically independent; in particular $\zeta(\alpha p)$ is transcendental.
\null\hfill $ \square$
\end{proof}

%-------------------------------------------
\bibliographystyle{plain}

\begin{thebibliography}{BPSS22}
 
\def\doi#1{\href{https://doi.org/#1}{doi #1}}
\def\zbl#1{\href{http://msp.org/idx/zbl/#1}{Zbl #1}} 
\def\MR#1{\href{https://mathscinet.ams.org/mathscinet/relay-station?mr=#1}{MR #1}}
 
 \bibitem[B02]{B02} 
 C. Bertolin.
 \newblock {\em P\'eriodes de 1-motifs et transcendence.}
 \newblock J. Number Theory {\bf 97} no. 2, 204--221 (2002).
\\ 
\doi{10.1016/S0022-314X(02)00002-1}
\quad
\zbl{1067.11041} 
\quad
\MR{1942957}

 \bibitem[B20]{B20}
 C. Bertolin.
 \newblock {\em Third kind elliptic integrals and 1-motives. }
 With a letter of Y. Andr\'e and an appendix by M. Waldschmidt.
 \newblock J. Pure Appl. Algebra {\bf 224}, no.10, Article ID 106396 (2020).
\\
\doi{10.1016/j.jpaa.2020.106396}
\quad
\zbl{1450.11077}
\qquad
\MR{4093081}


 \bibitem[B25]{BFutur}
 C. Bertolin.
 \newblock {\em The geometrical origin of conjectures \`{a} la Schanuel, }
 \newblock Work in progress. 

 \bibitem[BP24]{BP}
 C. Bertolin \&\ P. Philippon.
 \newblock {\em Mumford--Tate groups of 1-motives and Weil pairing. }
 \newblock J. Pure Appl. Algebra {\bf 228}, no. 10 Article ID 107702 (2024). 
 \\
 \doi{10.1016/j.jpaa.2024.107702}
 \qquad
\zbl{07854091}
\quad
\MR{4741463}
 
 \bibitem[BPSS22]{BPSS}
 C. Bertolin, P. Philippon, B. Saha \&\ E. Saha.
 \newblock {\em Semi--abelian analogues of Schanuel Conjecture and applications.}
 \newblock J. Algebra {\bf 596}, 250--288 (2022). 
 \\
 \doi{10.1016/j.jalgebra.2021.12.040}
 \qquad
\zbl{1495.11085}
\qquad
\MR{4369163}

 \bibitem[BW26]{BWFutur}
 C. Bertolin \&\ M. Waldschmidt. 
\newblock {\em Variations on Schanuel's Conjecture for elliptic and quasi--elliptic functions II: the non--split case. }
\newblock Work in progress. 
 
 \bibitem[BrK77]{BK}
 W.D. Brownawell \&\ K.K. Kubota. 
\newblock {\em The algebraic independence of Weierstrass functions and some related numbers. }
\newblock Acta Arith. {\bf 33}, 111--149 (1977). 
\\
\doi{10.4064/aa-33-2-111-149}
\qquad
\zbl{0356.10027}
\qquad
\MR{0444582}
 
\bibitem[Cha85]{Cha}
K. Chandrasekharan.
\newblock {\em Elliptic functions. }
\newblock Grundlehren der Mathematischen Wissenschaften, {\bf 281}. 
Springer--Verlag 
(1985). 
\\ 
\doi{10.1007/978-3-642-52244-4}\qquad
\zbl{0575.33001}
\qquad
\MR{0808396}

 \bibitem[Chu76]{Ch}
G.V. Chudnovsky. 
\newblock {\em Contributions to the theory of transcendental numbers. }
\newblock Mathematical Surveys and Monographs, no. {\bf 19}. Providence, R.I.: American Mathematical Society. 
(1984). 
\\
\url{https://bookstore.ams.org/view?ProductCode=SURV/19}\ 
 \zbl{0594.10024}
\ 
\MR{0772027}
 
 \bibitem[Co89]{Cox}
 D. Cox.
 \newblock {\em Primes of the form $x^2+ny^2$.
 Fermat, class field theory and complex multiplication. }
 \newblock New York etc.: John Wiley \& Sons (1989).
 \\
 \doi{10.1090/chel/387}
 \qquad
\zbl{1504.11002}
\qquad
\MR{4502401}


 \bibitem[F22]{F} 
 R. Fricke.
 \newblock {\em Die elliptischen Funktionen und ihre Anwendungen.}
 \newblock Vol. II, Teubner, Leipzig--Berlin, (1922).
 \\
 \doi{10.1007/978-3-642-19561-7}
 \qquad
\zbl{1230.01037}
\qquad
\MR{3221641}

 
 \bibitem[K23]{K}
 K. Senthil Kumar.
 \newblock 
 {\em On the values of Weierstrass zeta and sigma functions (With an appendix by D. Masser).}
 \newblock Acta Arith.{\bf 208}, no. 3, 285--294 (2023).
 \\
 \doi{10.4064/aa230201-22-5}
 \qquad
\zbl{07732772}
\qquad
\MR{4631984}
 
 
 \bibitem[L66]{LangITN}
S. Lang. 
\newblock 
{\em Introduction to transcendental numbers. }
\newblock Addison--Wesley Series in Mathematics. 
Reading, Mass., (1966). 
\\
ISBN 9780201041767 \qquad %https://books.google.fr/books?id=1m2NtAEACAAJ
\zbl{0144.04101}
\qquad
\MR{0214547}
\\
\newblock 
Reprint, 
Collected papers. Volume I: 1952--1970.
New York, NY: Springer 
(2000). 
\\ 
\url{https://link.springer.com/book/9781461461364}\qquad
 \zbl{0955.01032}
\qquad
\MR{3087336}
 
 \bibitem[L78]{LangECDA}
 S. Lang. 
 \newblock 
{\em Elliptic curves: Diophantine analysis. }
\newblock 
Grundlehren der Mathematischen Wissenschaften. {\bf 231}. 
Springer--Verlag. 
(1978). 
\\ 
\doi{10.1007/978-3-662-07010-9} \qquad
\zbl{0388.10001}
\qquad
\MR{0518817}

 \bibitem[L87]{LangEC}
 S. Lang. 
 \newblock 
{\em Elliptic functions.} Second edition. 
\newblock 
Graduate Texts in Mathematics, {\bf 112} 
Springer--Verlag. 
(1987).
\\ 
\doi{10.1007/978-1-4612-4752-4}\qquad
\zbl{0615.14018}
\qquad
\MR{0890960}
 
\bibitem[MOS66]{MOS} W. Magnus, F. Oberhettinger \&\ R. P. Soni.
\newblock {\em Formulas and theorems for the special functions of mathematical physics.}
\newblock 3rd enlarged ed. 
 Die Grundlehren der mathematischen Wissenschaften. {\bf 52}. Berlin--Heidelberg--New York: Springer--Verlag (1966).
\\
\doi{10.1007/978-3-662-11761-3}
\qquad
\zbl{0143.08502}
\qquad
\MR{0232968}
 
 \bibitem[M75]{M75} 
 D. W. Masser.
 \newblock {\em Elliptic functions and transcendence.}
 \newblock Lecture Notes in Mathematics, Vol. {\bf 437}. Springer--Verlag, Berlin--New York (1975).
\\ 
\doi{10.1007/BFb0069432}\qquad
\zbl{0312.10023}
\qquad
\MR{0379391}

 \bibitem[M16]{M}
 D. W. Masser.
 \newblock {\em Auxiliary polynomials in number theory.}
 \newblock Cambridge Tracts in Mathematics {\bf 207}. Cambridge University Press (2016).
 \\
 \doi{10.1017/CBO9781107448018}
 \qquad
\zbl{1354.11002}
\qquad
\MR{3497545}
 
 \bibitem[P83]{P83}
 P. Philippon.
 \newblock 
 {\em Vari\'{e}t\'{e}s ab\'{e}liennes et ind\'{e}pendance alg\'{e}brique II: un analogue ab\'{e}lien du th\'{e}or\`{e}me de Lindemann--Weierstra\ss.}
 \newblock Invent. math. {\bf 72}, 389--405 (1983).
 \\
 \doi{10.1007/BF01398395}
\qquad
\zbl{0516.10026}
\qquad
\MR{0704398}

\bibitem[Ra68]{Ramachandra}
\newblock 
K. Ramachandra.
 {\em Contributions to the theory of transcendental numbers. I, II.}
 \newblock 
Acta Arith. {\bf 14}, 65--72 (1968); {\bf 14}, 73--88 (1968). 
\\
\doi{10.4064/aa-14-1-65-72}\quad 
\zbl{0176.33101}
\qquad
\MR{0224566}

\bibitem[Re82]{Reyssat82}
E. Reyssat.
\newblock {\em Propri\'{e}t\'{e}s d'ind\'{e}pendance alg\'{e}brique de nombres li\'{e}s aux fonctions de Weierstrass.}
\newblock Acta Arith. {\bf 41}, 291--310 (1982).
 \\
 \doi{10.4064/aa-41-3-291-310}
\qquad
\zbl{0491.10026}
\qquad
\MR{0668915}
 
 \bibitem[Sc37]{Schneider} 
 Th. Schneider.
 \newblock {\em Arithmetische Untersuchungen elliptischer Integrale. }
 \newblock Math. Annalen {\bf 113}, 1--13 (1937). 
 \\
 \doi{10.1007/BF01571618}
 \qquad
\zbl{0014.20402}
\qquad
\MR{1513075}
 
 \bibitem[Sc57]{SchneiderLivre} 
 Th. Schneider.
 \newblock {\em Einf\"{u}hrung in die transzendenten Zahlen. }
 \newblock Die Grundlehren der Mathematischen Wissenschaften. Band 81. Berlin--G\"{o}ttingen--Heidelberg: Springer--Verlag (1957).
\\ 
\doi{10.1007/978-3-642-94694-3}\qquad
\zbl{0077.04703}
\qquad
\MR{0086842}
 
 
 \bibitem[Si86]{Silverman} 
 J.H. Silverman.
 \newblock 
{\em The arithmetic of elliptic curves.} 2nd ed. 
\newblock 
Graduate Texts in Mathematics {\bf 106}. New York, NY: Springer. 
(2009). 
\\ 
\doi{10.1007/978-0-387-09494-6}\qquad
\zbl{1194.11005}
\qquad
\MR{2514094}

 \bibitem[Wal79]{Wald79} 
 M. Waldschmidt.
	\newblock 
{\em Nombres transcendants et fonctions sigma de Weierstrass. }
	\newblock 	\href{https://mathreports.ca/article/nombres-transcendants-et-fonctions-sigma-de-weirstrass}{C. R. Math. Rep. Acad. Sci. Canada} {\bf 1} (1978/79), no. 2, pp. 111--114. 
\\
\MR{0519536}
 
 
 
 \bibitem[Wal25]{W_Schanuel}
 M. Waldschmidt.
 \newblock 
 {\em Schanuel property for elliptic and quasi--elliptic functions. }
 \newblock Work in progress. 

\bibitem[Was08]{Washington} 
L. C. Washington. 
\newblock 
{\em Elliptic Curves: Number Theory and Cryptography,} 1rst edition 2003. 
\newblock 2nd ed. (Discrete Mathematics and Its Applications) 
Boca Raton, FL: Chapman and Hall/CRC (2008).
\\
\href{https://www.routledge.com/Elliptic-Curves-Number-Theory-and-Cryptography-Second-Edition/Washington/p/book/9781420071467}{ISBN 9781420071467}\qquad
\zbl{1200.11043}
\qquad
\MR{2404461}

\bibitem[WW27]{WW} 
E. T. Whittaker \&\ G. N. Watson.
\newblock 
{\em A course of modern analysis. An introduction to the general theory of infinite processes and of analytic functions; with an account of the principal transcendental functions }(1927). Fifth edition 2021.
\newblock Cambridge Mathematical Library. Cambridge University Press.
\\
\doi{10.1017/CBO9780511608759} 
 \qquad
\zbl{0951.30002}
\qquad
\MR{4286926}
 
 \bibitem[W\"u83]{W83}
G. W\"ustholz.
\newblock
{\em \"Uber das Abelsche Analogon des Lindemannschen Satzes I.}
\newblock Invent. math. {\bf 72}, 363--388 (1983).
\\
\doi{10.1007/BF01398393}
\qquad
\zbl{0528.10024}
\qquad
\MR{0704397}

\end{thebibliography}

\authoraddresses{
Cristiana Bertolin\\
Dipartimento di Matematica, Universit\`a di Padova, Via Trieste 63, Padova\\
\email cristiana.bertolin@unipd.it

Michel Waldschmidt\\
Sorbonne Universit\'e, Universit\'e Paris Cit\'e, CNRS, IMJ-PRG, F-75005 Paris, France\\
\email michel.waldschmidt@imj-prg.fr 
}

%%%%%%%%%%%
\end{document}